\documentclass[12pt]{amsart}
\usepackage{amsfonts,amssymb,amscd,amsmath,enumerate,verbatim,calc}
\usepackage{hyperref}
\usepackage[mathscr]{euscript}
\usepackage{ mathrsfs }
\usepackage[left=2.25cm,right=2.25cm,top=2.5cm,bottom=2.5cm]{geometry}

\usepackage{graphicx}
\usepackage{color}

\usepackage[all,cmtip]{xy}

\newcommand{\n}{\mathfrak{n} }
\newcommand{\m}{\mathfrak{m} }

\newcommand{\q}{\mathfrak{q} }
\newcommand{\p}{\mathfrak{p} }

\providecommand{\D}{{\mathcal D}}

\newcommand{\Z}{\mathbb{Z} }
\newcommand{\N}{\mathbb{N} }

\newcommand{\im}{\operatorname{image}}

\newcommand{\Ass}{\operatorname{Ass}}

\newcommand{\chars}{\operatorname{char}}

\newcommand{\ann}{\operatorname{ann}}
\newcommand{\hgt}{\operatorname{height}}

\newcommand{\supp}{\operatorname{supp}}
\newcommand{\MaxSpec}{\operatorname{MaxSpec}}
\newcommand{\Spec}{\operatorname{Spec}}

\newcommand{\injdim}{\operatorname{injdim}}

\newcommand{\Hom}{\operatorname{Hom}}
\newcommand{\Ext}{\operatorname{Ext}}
\newcommand{\Tor}{\operatorname{Tor}}

\theoremstyle{plain}

\newtheorem{theorem}{Theorem}[section]
\newtheorem{corollary}[theorem]{Corollary}
\newtheorem{lemma}[theorem]{Lemma}
\newtheorem{proposition}[theorem]{Proposition}
\newtheorem{question}[theorem]{Question}

\theoremstyle{definition}
\newtheorem{definition}[theorem]{Definition}
\newtheorem{s}[theorem]{}
\newtheorem{remark}[theorem]{Remark}

\newtheorem{example}[theorem]{Example}
\newtheorem{note}[theorem]{Note}

\theoremstyle{remark}
\newtheorem*{example*}{\it Example}
\newtheorem*{note*}{\it Note}
\newtheorem*{claim*}{\it Claim}

\newtheorem*{case*}{\it Case}

\title[Graded components of local cohomology modules]{Graded components of local cohomology modules supported on $\mathfrak{C}$-monomial ideals}
\date{\today}
\thanks{The second author is grateful to the Infosys Foundation for providing partial financial support.}

\author{Tony J. Puthenpurakal}
\address{Department of Mathematics, Indian Institute of Technology Bombay,
Mumbai 400076, India}
\email{\href{mailto:tputhen@math.iitb.ac.in}{tputhen@math.iitb.ac.in}}
\author{Sudeshna Roy}
\address{Department of Mathematics, Chennai Mathematical Institute,
Kelambakkam 603103, India}
\email{\href{mailto:sudeshnaroy.11@gmail.com}{sudeshnaroy.11@gmail.com}}

\usepackage[switch]{lineno}
\nolinenumbers

\begin{document}

\begin{abstract}
Let $A$ be a Dedekind domain of characteristic zero such that its localization at every maximal ideal has mixed characteristic with finite residue field. Let $R=A[X_1,\ldots, X_n]$ be a polynomial ring and $I=(a_1U_1, \ldots, a_c U_c)\subseteq R$ an ideal, where $a_j \in A$ (not necessarily units) and $U_j$'s are monomials in $X_1, \ldots, X_n$. We call such an ideal $I$ as a $\mathfrak{C}$-monomial ideal. Consider the standard multigrading on $R$. We produce a structure theorem for the multigraded components of the local cohomology modules $H^i_I(R)$ for $i \geq 0$. We further analyze the torsion part and the torsion-free part of these components. We show that if $A$ is a PID then each component can be written as a direct sum of its torsion part and torsion-free part. As a consequence, we obtain that their Bass numbers are finite.
\end{abstract}
\maketitle

\section{introduction}
Since its introduction by A. Grothendieck in 1961, local cohomology modules have become a fundamental tool in commutative algebra and algebraic geometry. The nonfinite generation of nonzero local cohomology modules makes them difficult to study. A great deal of research has been done by commutative algebraists and algebraic geometers to decode the structure and behaviour of these modules. Some finiteness properties over regular rings were observed by C. Huneke and R. Y. Sharp \cite{HunSha} in prime characteristic using the Frobenius action on the local cohomology modules. G. Lyubeznik \cite{Lyu93} obtained all but one finiteness results over any regular ring containing a field of characteristic zero based on his breakthrough observation that local cohomology modules on complete regular local rings containing a field of characteristic zero are enriched with an additional structure of a ﬁnitely generated module over the ring of differential operators. Over the last few years, the D-module structure of local cohomology modules turns up as a major tool in their study over a ring containing a field. Some developments have been made for unramified and particular ramified regular local rings of mixed characteristic by using the theory of D-modules, see \cite{Lyu2000}, \cite{Bet13}, \cite{BBLSZ}. Unfortunately local cohomology modules may fail to have similar finiteness properties over singular rings. The local cohomology modules $H^i_I(R)$ supported on monomial ideals $I$ in a polynomial ring $R:=k[X_1, \ldots,X_n]$ over a field $k$, gained interest of several researchers mainly for its connection with toric varieties. Besides, such modules are relatively simpler to handle and helps to construct examples or nonexamples. The modules $H^i_I(R)$ have natural $\Z^n$-graded structures and can be approached using combinatorial techniques. However, there is a dictionary between the D-module and the $\Z^n$-graded structure of $H^i_I(R)$, see \cite{Mon_LC}. It is well-known that $H^i_I(R)_{\underline{u}}$ are finite dimensional $k$-vector spaces for all $\underline{u} \in \Z^n$. Readers are encouraged to look through the survey article \cite{Mon_LC} for an highlight of recent progress in studying such $H^i_I(R)$.

Let $R = \bigoplus_{u \geq 0} R_u$ be a standard graded Noetherian ring, and $R_+ = \bigoplus_{u > 0} R_u$ be its irrelevant ideal. Let $N = \bigoplus_{u \in \Z} N_u$ be a finitely generated graded $R$-module. It is long known that $H^i_{R_+}(N)_u$ is finitely generated for all $u \in \Z$ and vanishes for all large values of $u$. Such modules are extremely useful in the study of Hilbert coefficients, Rees algebras, associated graded rings, and several others. It is worthwhile to remark that components of a graded local cohomology module rarely carry similar properties when its support is an arbitrary homogeneous ideal.  The first author \cite{TP2} did a comprehensive study of the components of $H^i_I(R)$ where $I$ is an arbitrary homogeneous ideal of a standard graded polynomial ring $R = R_0[X_1, \ldots,X_n]$ over a regular ring $R_0$ containing a field of characteristic zero. He tried out some questions related to the asymptotic behavior of graded local cohomology modules raised by M. Brodmann in \cite{Brod1} and proved that certain numerical invariants associated to $H^i_I(R)$ are asymptotically stable. In \cite{TPSR1}, \cite{TPSR2}, we established some of these results under relaxed assumptions on the base ring $R_0$. However, these results are not valid in general even for local cohomology modules supported on irrelevant ideals.

Let $A$ be a Dedekind domain of characteristic zero such that its localization at every maximal ideal has mixed characteristic with finite residue field. Let $K$ be its field of fractions. Notice that the ring of integers in a number field is an example of such ring. In \cite[Theorem 4.3]{BBLSZ}, B. Bhatt et al. proved that the set of associated primes of $H^i_I(R)$ as an $A$-module is finite. In light of the above discussions, we are now interested in analyzing the structure of multigraded components of local cohomology modules supported on monomial ideals $I$ in a
polynomial ring $R=R_{\underline{0}}[X_1, \ldots,X_n]$ with $R_{\underline{0}}=A$ and $\deg X_j=e_j \in \N^n$ for $j=1, \ldots, n$. But, in this case $H^i_I(R)_{\underline{u}}$ are finitely generated $A$-modules for all $\underline{u} \in \Z^n$ as we see in Proposition \ref{fin-gen}. Therefore, we turn our attention to a slightly more general ideal.

\begin{definition}\label{C_monomial}
We say that an ideal in $R$ is a \emph{$\mathfrak{C}$-monomial ideal} if it can be generated by some elements of the form $aU$ where $a \in A$ (possibly nonunit) and $U=X_1^{\xi_{1}} \cdots X_n^{\xi_{n}}$ for some $\xi_{j} \neq 0$.
\end{definition}

We point out that a $\mathfrak{C}$-monomial ideal is the "
usual monomial ideal'' (see \cite[Deﬁnition 1.1.1]{HerHib}) if all the coefficients $a$ occurring in a generating set are units.
If $I$ is a $\mathfrak{C}$-monomial ideal, components of $H^i_I(R)$ are not necessarily finitely generated, see Examples 6.1 and 6.2. Observe that $A$ is locally a one dimensional regular local ring. One of our main achievements in this work is a structure theorem for the graded components of $H^i_I(R)$. To study the torsion-free part of $H^i_I(R)_{\underline{u}}$, we look into the module $H^i_I(R) \otimes_{A} K$ in an obvious manner. If we set $S:=K[X_1, \ldots, X_n]$, then $H^i_I(R) \otimes_{A} K \cong H^i_{IS}(S)$ whose structure is mostly known. It is in fact easily seen that $IS \subseteq S$ is a usual monomial ideal. Throughout this article, the study will be carried out under the following setup.

\s {\bf Standard assumption.}\label{sai} Let $A$ be a Dedekind domain of characteristic zero such that for each height one prime ideal $\p$ in $A$, the local ring $A_{\p}$ has mixed characteristic with finite residue field. Suppose that $R=A[X_1, \ldots, X_n]$ is a polynomial ring with $\deg X_j=e_j \in \N^n$. Let $I \subseteq R$ be a $\mathfrak{C}$-monomial ideal
(see Definition \ref{C_monomial}) with a generating set $\{a_1U_1, \ldots, a_c U_c\},$ where $a_j \in A$ and $U_j$'s are monomials in $X_1, \ldots, X_n$. We set $M:=H^i_I(R)=\bigoplus_{\underline{u} \in \Z^n} M_{\underline{u}}$ and put $\mathcal{Z}:=\Ass_A H^i_I(R) \backslash \{0\}$.

\vspace{0.15cm}
Fix $\underline{u} \in \Z^n$. Recall that $\mathcal{Z}$ is a finite set due to \cite[Theorem 4.3]{BBLSZ}. Besides, in a one-dimensional domain any two nonzero prime ideals are coprime. These two facts permit us to express the torsion part of $H^i_I(R)_{\underline{u}}$ as a direct sum of its $\p$-torsion submodules, where $\p \in \mathcal{Z}$. Since $k:=A/\p \cong A_\p/\p A_\p$ is a finite field of prime characteristic by our hypotheses, we are able to utilize the results known over $\overline{R}:=R/\p R \cong k[X_1, \ldots, X_n]$.

In this article, we prove the following.

\begin{theorem}\label{intro}
Assume the hypothesis as in \ref{sai}. Take any prime ideal $\q$ in $A$. We set $T=\widehat{A_\q}$ and put $N=M \otimes_A T=H^i_{IS}(S)$, where $S=T[X_1, \ldots, X_n]$. Let $K$ and $K_\q$ denote the fraction fields of $R$ and $T$, respectively. Fix $\underline{u} \in \Z^n$. Then the following holds.
\begin{enumerate}[\rm (1)]
 \item  Let $t(M_{\underline{u}})$ denote the torsion submodule of $M_{\underline{u}}$. Then $t(M_{\underline{u}})= \oplus_{\p \in \mathcal{Z}} \Gamma_{\p}(M_{\underline{u}})$.

\item For each $\p \in \mathcal{Z}$, there are
finite numbers $\ell(\underline{u}),\beta_j(\underline{u}), t(\underline{u})$ such that
\[\Gamma_\p\left(M_{\underline{u}}\right)=E_A\left(\frac{A}{\p}\right)^{\ell(\underline{u})} \oplus \left(\bigoplus_{j=1}^{t(\underline{u})} \frac{A}{\p^{\beta_j(\underline{u})}A}\right).\]

\item Let $\q$ be a nonzero prime ideal in $A$ such that $a_i \notin \q$ for all $i=1, \ldots, c$. Then $E_A(A/\q)$ cannot appear as a direct summand of $t(M_{\underline{u}})$.

\item Write $\overline{N_{\underline{u}}}:=\frac{N_{\underline{u}}}{\Gamma_\q(N_{\underline{u}})}$. Then
$\overline{N_{\underline{u}}}=T^{a(\underline{u})} \oplus K_\q^{b(\underline{u})}$
for some finite numbers $a(\underline{u}), b(\underline{u}) \geq 0$.

\item $M_{\underline{u}} \otimes_A K \cong K^{\alpha(\underline{u})}$ for some finite $\alpha(\underline{u}) \geq 0$. Moreover, $\{\alpha(\underline{u}) \mid \underline{u} \in \Z^n\}$ is a finite set.

\item $\overline{M_{\underline{u}}}:=\frac{M_{\underline{u}}}{t(M_{\underline{u}})}$ is a flat $A$-module.

\item
\begin{enumerate}[\rm (a)]
 \item The short exact sequence
\begin{equation}\label{eq_ses_intro}
0 \to t\left(M_{\underline{u}}\right) \to M_{\underline{u}} \to \overline{M_{\underline{u}}} \to 0
\end{equation}
 splits
if $I$ is a usual monomial ideal or $A$ is a PID.
 \item The induced short exact sequence $0 \to t\left(M_{\underline{u}}\right)_{\q} \to \left(M_{\underline{u}}\right)_{\q} \to \left(\overline{M_{\underline{u}}}\right)_\q \to 0$ splits for every prime ideal $\q$ in $A$.
\end{enumerate}
\end{enumerate}
\end{theorem}
As a sequel of Theorem \ref{intro} (7), we show that the Bass numbers $\mu_j(\p, M_{\underline{u}})$ are finite for all primes $\p$ in $A$ and $\underline{u} \in \Z^n$.
But, we don't know whether
\eqref{eq_ses_intro} splits
when $A$ is not a PID.

\section{Understanding the torsion of the components}

We first recollect the framework for this study.

\s {\bf Setup.}\label{sa} Let $A$ be a Dedekind domain such that for each height one prime ideal $\p$ in $A$, the local ring $A_{\p}$ has mixed characteristic with finite residue field. Suppose that $R=A[X_1, \ldots, X_n]$ is a polynomial ring with $\deg a=\underline{0} \in \N^n$ for all $a$ in $A$ and $\deg X_j=e_j \in \N^n$ for $j=1, \ldots, n$. Let $I=(a_1U_1, \ldots, a_cU_c)$ be a $\mathfrak{C}$-monomial ideal in $R$ (see Definition \ref{C_monomial}), where $a_j$'s are nonzero elements in $A$ and $U_j$'s are monomials in $X_1, \ldots, X_n$.
In view of the induced natural $\Z^n$-grading on $H^i_I(R)$, we set $M:=H^i_I(R)=\bigoplus_{\underline{u} \in \Z^n} M_{\underline{u}}$. Further denote
\begin{equation}\label{assPrime}
\mathcal{Z}:=\{\p \mid \p \in \Ass_A M, \p \neq 0\}=\{\p_1, \ldots, \p_r\}.
\end{equation}
Apparently, $\Ass_A M_{\underline{u}} \subseteq \Ass_A M=\bigcup_{\underline{u} \in \Z^n}\Ass_{A} M_{\underline{u}}$ is finite for each $\underline{u} \in \Z^n$. Observe that $\Ass_{A} M$ contains the zero ideal when $M$ has some torsion-free element.

\vspace{0.15cm}
\noindent
{\it Notation}. Throughout this article, we refer the prime divisors of any ideal $J$ of $R$ by the associated primes of the $R$-module $R/J$, denoted by $\Ass_R R/J$.

\vspace{0.15cm}
We first discuss a few basic properties of the torsion submodule which are mostly well-known.
For the convenience of the reader, we provide the proofs.

Let $J$ be an ideal in $A$ and $\Gamma_J(-)$ denote the $J$-torsion functor, which is defined by $\Gamma_J(N)=\bigcup_{n \geq 1}(0:_N J^n)$ for any $A$-module $N$.

\begin{lemma}\label{gam_rel}
Let $I, J$ be comaximal ideals in $A$ and $N$ be an $A$-module. Then $\Gamma_{IJ}(N)=\Gamma_I(N) \oplus \Gamma_J(N)$.
\end{lemma}

\begin{proof}
Consider the \emph{Mayer-Vietoris sequence}
 \[0 \to H^0_{I+J}(N) \to H^0_{I}(N) \oplus H^0_{J}(N) \to H^0_{I \cap J}(N) \to H^1_{I+J}(N) \to \cdots.\]
 Since $I, J$ are comaximal ideals, $H^i_{I+J}(N)=0$ for all $i \geq 0$. Particularly, $H^0_{I+J}(N)=0= H^1_{I+J}(N)$. Moreover, $I \cap J=IJ$. Thus $\Gamma_{IJ}(N) \cong \Gamma_{I}(N) \oplus \Gamma_{J}(N)$.
 \end{proof}

 \begin{lemma}
Let $N$ be an $A$-module. Then $\Gamma_\p(N)$ is an $A_\p$-module for each nonzero prime ideal $\p$ in $A$.
 \end{lemma}

 \begin{proof}
  It is enough to show that any element $a \in A \backslash \p$ acts as an unit on $\Gamma_\p(N)$. Observe that
\[\left(\frac{\Gamma_\p(N)}{a \Gamma_\p(N)}\right)_\q\cong \frac{\Gamma_\p(N)_\q}{\left(a \Gamma_\p(N)\right)_\q}=0\]
for any prime ideal $\q$ in $A$. In fact, $\left(a \Gamma_\p(N)\right)_\p \cong \Gamma_\p(N)_\p$ as $a \notin \p$. Whereas, $\Gamma_\p(N)_\q=0$ for each $\q \neq \p$. Hence $\Gamma_\p(N) \cong a \Gamma_\p(N)$.
 \end{proof}

 \begin{remark}\label{completion}
  Let $\widehat{A_\p}$ denote the completion of $A_\p$ with respect to $\p A_\p$. Observe that $\Gamma_\p(N)$ has a natural $\widehat{A_\p}$-module structure. For instance take $u \in \Gamma_\p(N)$. If $\p^r u=0$, then for any $a \in \widehat{A_\p}$, choose $b \in A_\p$ such that $a-b \in \left(\p A_\p\right)^{r}$. It is easy to verify that the action $a \cdot u=b \cdot u$ gives a natural $\widehat{A_\p}$-module structure on $N$. We note that if $a-c \in \left(\p A_\p\right)^{r}$ for some $b \neq c \in A_\p$, then $b-c=(a-c)-(a-b) \in \left(\p A_\p\right)^{r}$. So $(b-c)u=0$, i.e., $bu=cu$. Thus the action is well-defined.
 \end{remark}

\begin{proposition}\label{tor_sub_fun}
Let $\mathcal{Z}$ be as in \eqref{assPrime}. Fix $\underline{u} \in \Z^n$. Take the torsion submodule $t(M_{\underline{u}})$ of $M_{\underline{u}}$. Then \[t(M_{\underline{u}})=
\bigoplus_{\p \in \mathcal{Z}} \Gamma_{\p}(M_{\underline{u}}).\]
\end{proposition}

\begin{proof}
Notice that $\Ass_A t\left(M_{\underline{u}}\right) \subseteq \Ass_A M_{\underline{u}} \subseteq \Ass_A M$. Because $\p_1\cdots\p_{r-1}$ and $\p_r$ are comaximal ideals, using Lemma \ref{gam_rel} and an easy induction on $r$, we get that
\[\Gamma_{\p_1\cdots\p_r}(N) \cong \bigoplus_{i=1}^r \Gamma_{\p_i}(N).\]
Therefore, it is enough to show that $t\left(M_{\underline{u}}\right)=\Gamma_{\p_1\cdots\p_r}(M_{\underline{u}})$. For any $m \in \Gamma_{\p_1\cdots\p_r}(M_{\underline{u}})$,
\[(\p_1\cdots\p_r)^b m=0\] for some $b \geq 1$. Hence $\Gamma_{\p_1\cdots\p_r}(M_{\underline{u}}) \subseteq t(M_{\underline{u}})$. Now take $z \in t(M_{\underline{u}})$ and set $J:=\ann_A z$. Suppose that $\Ass_A A/J=\{\q_1, \ldots, \q_t\}$, i.e., $\sqrt{J}=\bigcap_{i=1}^t \q_i$. As $Az \cong A/J$ so \[\Ass_A A/J = \Ass_A Az \subseteq \Ass_A M.\]
Since $J \neq 0$, the zero ideal does not belong to $\Ass_A A/J$. Thus $\{\q_1, \ldots, \q_t\} \subseteq \{\p_1 \ldots, \p_r\}$.
As $\q_i$'s are comaximal ideals in $A$ so we have $\bigcap_{i=1}^t \q_i=\prod_{i=1}^t \q_i$ and hence
\[\left(\prod_{i=1}^t \q_i\right)^w=\left(\sqrt{J}\right)^w \subseteq J\]
for some $w \geq 1$. Therefore, $\left(\prod_{i=1}^r \p_i\right)^w \cdot z=0$ which implies that $z \in \Gamma_{\p_1\cdots\p_r}(M_{\underline{u}})$.
\end{proof}

Let $T$ be a PID and $N$ be a finitely generated $T$-module. By the \emph{structure theorem of finitely generated   `modules over a PID},
\[N \cong T^s \oplus T/(\q_1^{\alpha_1}) \oplus T/(\q_2^{\alpha_2}) \oplus \cdots T/(\q_t^{\alpha_t})\]
for $s, \alpha_i \geq 0$ and prime ideals $\q_1, \ldots, \q_t$ in $T$.

\begin{remark}
 There is a typo in \cite[Theorem 23.2 ii)]{Mat}. The correct statement is: \emph{if $\varphi: A \to B$ is a ring homomorphism and $B$ is flat over $A$, then for any $A$-module $M$
\[\Ass_B (M \otimes_A B)=\bigcup_{\p \in \Ass_A M} \Ass_B \left(\frac{B}{\p B}\right).\]}
Therefore, if $A, M$ are as in \ref{sa} and $B=\widehat{A_\p}$ for some $\p \in \mathcal{Z}$, then $\Ass_B (M \otimes_A B)=\{\p B\}$.
\end{remark}

We now use it to describe the structure of graded components of certain local cohomology modules.
This is the main result of our article.

\begin{theorem}[Structure theorem]\label{struThm}
Let $\mathcal{Z}$ be as in \eqref{assPrime}. Fix $\p \in \mathcal{Z}$ and $\underline{u} \in \Z^n$. Then
\[\Gamma_\p\left(M_{\underline{u}}\right)=E_A\left(\frac{A}{\p}\right)^{\ell(\underline{u})} \oplus \left(\bigoplus_{j=1}^{t(\underline{u})} \frac{A}{\p^{\beta_j(\underline{u})}A}\right)\]
for some finite numbers $\ell(\underline{u}), \beta_j(\underline{u}), t(\underline{u})$.

Furthermore, fix a nonzero prime ideal $\q$ in $A$. Set $T=\widehat{A_\q}$ and put $N=M \otimes_A T=H^i_{IS}(S)$, where $S=T[X_1, \ldots, X_n]$. We denote $\overline{N_{\underline{u}}}:=\frac{N_{\underline{u}}}{\Gamma_\q(N_{\underline{u}})}$. Then
\[\overline{N_{\underline{u}}}=T^{a(\underline{u})} \oplus K_\q^{b(\underline{u})},\]
for some finite numbers $a(\underline{u}), b(\underline{u}) \geq 0$, where $K_\q$ denotes the quotient field of $T$.
\end{theorem}

\begin{remark}\label{tor-comp-rel}
 It is easily seen that $\Gamma_\p\left(N_{\underline{u}}\right)=\Gamma_\p\left(M_{\underline{u}}\right)$. For instance, as $H^0_{\p R}(H^i_I(R))$ is a $T$-module by Remark \ref{completion} so we get
 \[H^0_{\p R}(H^i_I(R)) \cong
\left(H^0_{\p R}(H^i_I(R)) \otimes_R R\right) \otimes_T T \cong H^0_{\p R}(H^i_I(R)) \otimes_R S \cong H^0_{\p S}(H^i_{IS}(S))\]
 by repetitively using the \emph{flat base change theorem}.
\end{remark}

\begin{proof}
We prove the statements in three parts. In the first part, we show that
\begin{equation}\label{main_1}
 \Gamma_\p\left(N_{\underline{u}}\right)=E_T\left(\frac{A}{\p}\right)^{\ell(\underline{u})} \oplus \left(\bigoplus_{j=1}^{t(\underline{u})} \frac{T}{\p^{\beta_j(\underline{u})}T}\right)
\end{equation}
for $\p \in \mathcal{Z}$ and some finite $\ell(\underline{u}), \beta_j(\underline{u}), t \geq 0$. In the second part, we prove that
\[\overline{N_{\underline{u}}}=T^{a(\underline{u})} \oplus K_\p^{b(\underline{u})}\]
for any nonzero prime ideal $\p$ in $A$ and some finite $a(\underline{u}), b(\underline{u}) \geq 0$. Finally, in the third part, we deduce the structure of the $A$-module $\Gamma_\p\left(M_{\underline{u}}\right)$
from \eqref{main_1}.

\vspace{0.15cm}
\noindent
{\it Part 1.} Notice that $N_{\underline{u}}=M_{\underline{u}} \otimes T$ and $T/\p T =\widehat{A_\p}/\p \widehat{A_\p} \cong A/\p$ is a field of prime characteristic. Moreover, $T$ is a complete DVR. Therefore, $\p T =(\pi)$ for some $\pi \in T$. We set $\overline{S}=S/\p S$. The short exact sequence
\[0 \to S \xrightarrow{ \cdot \pi} S \to \overline{S} \to 0\]
induces a long exact sequence
\begin{equation}\label{eq1}
\cdots \to H^{i-1}_I(\overline{S}) \xrightarrow{d_{i-1}} H^i_I(S) \xrightarrow{ \cdot \pi} H^i_I(S) \xrightarrow{\theta_i} H^i_I(\overline{S}) \to \cdots.
\end{equation}
 Observe that $T/\p T \cong B/\p B \cong A/\p$ is a field. As $S/\p S \cong \frac{A}{\p}[X_1, \ldots, X_n]$ so $H^i_I(\overline{S})_{\underline{u}}$ is a finite dimensional $A/\p$-vector space for all $i \geq 0$ and each $\underline{u} \in \Z^n$. We put $k=A/\p$. Consider the short exact sequence
\begin{equation}\label{eq2}
0 \to H_1(\p, N_{\underline{u}}) \to N_{\underline{u}} \xrightarrow{ \cdot \pi} N_{\underline{u}} \to H_0(\p, N_{\underline{u}}) \to 0.
\end{equation}
From \eqref{eq1} we have $H_1(\p, N_{\underline{u}})=\ker \left(N_{\underline{u}} \xrightarrow{ \cdot \pi} N_{\underline{u}}\right)=\left(\im d_{i-1}\right)_{\underline{u}}$ is a finite dimensional $k$-vector space. Suppose that $\dim_k H_1(\p, N_{\underline{u}})=d_0$. Applying $\Hom_T(k, -)$ to \eqref{eq2}, we get
\[0 \to \Hom_T(k, k^{d_0}) \to \Hom_T(k, N_{\underline{u}}) \xrightarrow{ 0} \Hom_T(k, N_{\underline{u}}).\]
Hence $\Hom_T(k, N_{\underline{u}})$ is a finite dimensional $k$-vector space.

Next, consider the short exact sequence
\begin{equation}\label{eq3}
0 \to \Gamma_\p(N_{\underline{u}}) \to N_{\underline{u}} \to \overline{N_{\underline{u}}} \to 0.
\end{equation}
Applying $\Hom_T(k, -)$ we get
\[0 \to \Hom_T(k, \Gamma_\p(N_{\underline{u}})) \to \Hom_k(k, N_{\underline{u}}).\]
So $\Hom_T(k, \Gamma_\p(N_{\underline{u}}))$ is a finite dimensional $k$-vector space. Since $T$ is a regular local ring of dimension one,
\[\injdim_{T}(\Gamma_\p(N_{\underline{u}}))\leq \dim T=1.\]
We put $E:=E_T(k)$. As $\Gamma_\p(N_{\underline{u}})$ is a $\p T$-torsion module so it has an injective resolution consisting of $\p$-torsion injective modules, i.e., each injective module which appears in the resolution is some power of $E$. Suppose that
\begin{equation}\label{p-inj}
 0 \to \Gamma_\p(N_{\underline{u}}) \to E^{d_1} \to E^{d_2} \to 0
\end{equation}
is a minimal injective resolution.  Applying $\Hom_T(k, -)$, we get
\[0 \to \Hom_T(k, \Gamma_\p(N_{\underline{u}})) \to \Hom_T(k,E^{d_1}) \xrightarrow{0} \Hom_T(k,E^{d_2}).\]
Recall that $\Hom_T(k,E) \cong E_k(k) \cong k$. Thus $k^{d_1} \cong \Hom_T(k,E^{d_1}) \cong \Hom_T(k, \Gamma_\p(N_{\underline{u}}))$.
So $d_1$ is finite and hence $d_2$ is finite. Here we would like to highlight that if $\injdim \Gamma_\p(N_{\underline{u}})=0$, then $d_2=0$ in \eqref{p-inj} and $\Gamma_\p(N_{\underline{u}}) \cong E^{d_1}$ for finite $d_1$.

\vspace{0.2cm}
From \cite[Theorem 18.6. v)]{Mat} and the inclusion map $\Gamma_\p(N_{\underline{u}}) \hookrightarrow E^{d_1}$ in \eqref{p-inj}, it follows that $\Gamma_\p(N_{\underline{u}})$ is Artinian. Therefore, by \cite[Theorem 3.2.13. c)]{BH}, $\Gamma_\p(N_{\underline{u}})^{\vee\vee} \cong \Gamma_\p(N_{\underline{u}})$, where $(-)^\vee:=\Hom_T(-, E)$ denotes the Matlis dual. Observe that the Artinianness of $\Gamma_\p(N_{\underline{u}})$ is not automatic, as $N_{\underline{u}}$ need not be a finitely generated $T$-module. Next, applying
$(-)^\vee$ to \eqref{p-inj}, we get
\[0 \to T^{d_2} \to T^{d_1} \to \Gamma_\p(N_{\underline{u}})^\vee \to 0.\]
It follows that $\Gamma_\p(N_{\underline{u}})^\vee$ is a finitely generated $T$-module. Since $T$ is a PID, by the \emph{structure theorem for finitely generated modules over a PID}, we get
\[\Gamma_\p(N_{\underline{u}})^\vee \cong T^{\ell(\underline{u})} \oplus \left( \bigoplus_{j=1}^{t(\underline{u})} \frac{T}{\p^{\beta_j(\underline{u})}T} \right).\]
for some finite numbers $\ell(\underline{u}), t(\underline{u})$. Further,
observe that
$\Hom_T(T/\p^{\beta_j(\underline{u})}T, E) \cong E_{T/\p^{\beta_j(\underline{u})} T}(k) \cong T/\p^{\beta_j(\underline{u})}T$,
as $T/\p^{\beta_j(\underline{u})}T$ is Gorenstein Artin. Hence
\begin{equation}\label{Gamma-Struc}
\Gamma_\p(N_{\underline{u}}) = E^{\ell(\underline{u})} \oplus \left( \bigoplus_{j=1}^{t(\underline{u})} \frac{T}{\p^{\beta_j(\underline{u})}T} \right).
\end{equation}

\vspace{0.15cm}
\noindent
{\it Part 2.} We set $\overline{N_{\underline{u}}}=N_{\underline{u}}/\Gamma_\p(N_{\underline{u}})$ and consider the following commutative diagram
\[
\xymatrix@C=0.5em@R=0.75em{
	0 \ar[r] & \Gamma_\p(N_{\underline{u}}) \ar[r] \ar[d]^\pi & N_{\underline{u}} \ar[r] \ar[d]^\pi &\overline{N_{\underline{u}}} \ar[r] \ar[d]^\pi & 0\\
0 \ar[r] & \Gamma_\p(N_{\underline{u}}) \ar[r] & N_{\underline{u}} \ar[r] \ar[d] &\overline{N_{\underline{u}}} \ar[r] \ar[d] & 0\\
&& H_0(\pi, N_{\underline{u}}) \ar[r] & \frac{\overline{N_{\underline{u}}}}{\p \overline{N_{\underline{u}}}} \ar[r]& 0} \]
Since $\dim_k H_0(\pi, N_{\underline{u}})< \infty$ by \eqref{eq1},
we get that $\dim_k \overline{N_{\underline{u}}}/\p \overline{N_{\underline{u}}}< \infty$.
We set
$L:=\bigcap_{j=1}^\infty \p^j \overline{N_{\underline{u}}}$.
Then $\dim_k (\overline{N_{\underline{u}}}/L)/\p( \overline{N_{\underline{u}}}/L)< \infty$, since we have $(\overline{N_{\underline{u}}}/L)/\p( \overline{N_{\underline{u}}}/L) \cong \overline{N_{\underline{u}}}/\p \overline{N_{\underline{u}}}$. Moreover, $\bigcap_{j=1}^\infty \p^j \left(\overline{N_{\underline{u}}}/L\right)=\left(\bigcap_{j=1}^\infty \p^j\overline{N_{\underline{u}}}\right)/L=0$.
Thus by \cite[Theorem 8.4]{Mat}, it follows that $\overline{N_{\underline{u}}}/L$ is finitely generated.
Next, we consider the short exact sequence
\[0 \to (0:_L \pi) \to L \xrightarrow{\cdot \pi} L \to C \to 0.\]

To see
$(0:_L \pi)=0$, take $a \in L$ such that $\pi a=0$. Suppose that $a=\pi^j\overline{b}$ for some $b \in N_{\underline{u}}$. Observe that $\pi a=\pi^{j+1} \overline{b}=0$. Therefore, $\pi^{j+1}b \in \Gamma_\p(N_{\underline{u}})$, i.e., $\pi^i(\pi^{j+1}b)=\pi^{i+j+1}b=0$ for some $i \geq 1$. Thus $b \in \Gamma_\p(N_{\underline{u}})$. Hence $\overline{b}=0$ and it follows that $a=0$.

We now claim that $C=0$. Let $z \in L$. As $L=\bigcap_{j=1}^\infty \p^j \overline{N_{\underline{u}}} \subseteq \p\overline{N_{\underline{u}}}$ so we have $z=\pi \alpha$ for some $\alpha \in \overline{N_{\underline{u}}}$. It is enough to show that $\alpha \in L$, i.e., $\alpha \in \p^j N_{\underline{u}}$ for all $j\geq 1$.
Since $L \subseteq \bigcap_{j=1}^\infty \p^{j+1} \overline{N_{\underline{u}}}$, we have $z=\pi\alpha \in \p^{j+1} \overline{N_{\underline{u}}}$ for each $j \geq 1$. Therefore, $\pi\alpha=\pi^{j+1} \xi$ for some $\xi \in \overline{N_{\underline{u}}}$. As $\pi$ is a regular element on $\overline{N_{\underline{u}}}$ so $\alpha=\pi^j \xi$. The claim follows.

Since $L \cong \pi L$, $\pi$ acts as an unit on $L$. Thus $L$ is a $K_\p$-module. Hence $L\cong K_\p^a$. We put $W=K_\p[X_1, \ldots, X_n]$. Then $H^i_{IS}(S) \otimes_T K_\p \cong H^i_{IW}(W)$ by the {\it Flat Base Change Theorem}. Thus $\overline{N_{\underline{u}}}\otimes K_\p \cong N_{\underline{u}} \otimes K_\p$ is a finite dimensional $K_\p$-vector space. As $L \subset \overline{N_{\underline{u}}}$ so we have $L \otimes_{K_\p} K_\p \subset \overline{N_{\underline{u}}}\otimes_{K_\p} K_\p$. It follows that $a$ is finite number.

From
the following commutative diagram
\[
\xymatrix@C=0.5em@R=1em{  &&0  \ar[d]&  & & \\
	0 \ar[r] & L \ar[r] \ar[d]^\pi & \overline{N_{\underline{u}}} \ar[r] \ar[d]^\pi &\overline{N_{\underline{u}}}/L \ar[r] \ar[d]^\pi & 0\\
0 \ar[r] & L \ar[r] \ar[d]& \overline{N_{\underline{u}}} \ar[r]  &\overline{N_{\underline{u}}}/L  \ar[r] & 0\\
&0 & &&}, \]
we get that the map $\overline{N_{\underline{u}}}/L \xrightarrow{\cdot \pi} \overline{N_{\underline{u}}}/L$ s injective, by the \emph{snake lemma}. So the module $\overline{N_{\underline{u}}}/L$ does not contain any $\p T$-torsion element. Since $\overline{N_{\underline{u}}}/L$ is a finitely generated $T$-module, by the structure theorem,
\[\overline{N_{\underline{u}}}/L\cong T^{b(\underline{u})}\]
for some finite number $b(\underline{u}) \geq 0$. From the short exact sequence
\begin{equation*}\label{int-con}
0 \to L \to \overline{N_{\underline{u}}} \to \overline{N_{\underline{u}}}/L \to 0
\end{equation*}
it follows that
\[\overline{N_{\underline{u}}} \cong K_\p^{a(\underline{u})} \oplus T^{b(\underline{u})}.\]
Next, take any nonzero prime ideal $\q \notin \mathcal{Z}$. Notice that $\Gamma_\q(N_{\underline{u}}) = 0$. 
Since $H_0(\q, N_{\underline{u}})$ is a finite dimensional $k = A/\q $-vector space, using the above method one can get the same expression for $\overline{N_{\underline{u}}}=N_{\underline{u}}/\Gamma_\q(N_{\underline{u}})=N_{\underline{u}}$.

\vspace{0.2cm}
\noindent
{\it Part 3.} From Remark \ref{completion}, we have a natural $T=\widehat{A_\p}$-module structure on $\Gamma_\p\left(M_{\underline{u}}\right)$. By \cite[Theorem 3.6]{Matlis},
\[E_T(T/\p T) \cong E_{B}(B/\p B) \cong E_A(A/\p).\]
Moreover, as $B/\p^{\beta_j(\underline{u})}B$ is an
Artinian ring so $T/\p^{\beta_j(\underline{u})}T \cong B/\p^{\beta_j(\underline{u})}B$. Again
\[\frac{B}{\p^{\beta_j(\underline{u})}B} \cong \left(\frac{A}{\p^{\beta_j(\underline{u})}}\right)_{\p} \cong \frac{A}{\p^{\beta_j(\underline{u})}}.\]
Thus the result follows from Part $1$.
\end{proof}

By $\MaxSpec(A)$, we denote the set of all maximal ideals in $A$.
\begin{proposition}\label{choicePrime}
Let $\q \in \MaxSpec(A)$ such that $a_j \notin \q$ for all $j=1, \ldots, c$.
Then $E_A(A/\q)$ can not appear as a direct summand of $t(M_{\underline{u}})$.
\end{proposition}

\begin{proof}
We set $\theta:=a_1\cdots a_c$. Since $a_i$'s are nonzero and $A$ is a domain, $\theta$ is also a nonzero element in $A$. Clearly, $a_i$'s are invertible in $R_\theta\cong A_\theta[X_1, \ldots, X_n]$ and hence $I_\theta=(X_1, \ldots, X_n)R_\theta$.
Furthermore, $H^i_I(R)_\theta \cong H^i_{I_\theta}(R_\theta)$,
whose graded components are finitely generated $A_\theta$-modules by Proposition \ref{fin-gen}. Now for any $\q \in \MaxSpec(A)$ such that $a_i \notin \q$ for $i=1, \ldots, c$, we have $\left(A_\theta\right)_\q \cong A_\q$. We set $W=A \backslash \q$. Clearly, $W$ is a multiplicatively closed subset of $R$ and $W^{-1}R \cong A_\q[X_1, \ldots, X_n]$. Therefore,
\[H^i_{I_\theta}(R_\theta)_{\underline{u}}=H^i_I(R)_{\underline{u}} \otimes_{A} A_{\theta} \quad \mbox{ and } \quad  H^i_{W^{-1}I}(W^{-1}R)_{\underline{u}} = H^i_{I_\theta}(R_\theta)_{\underline{u}} \otimes_{A_{\theta}} A_\q \]
for each $\underline{u} \in \Z^n$. Therefore, $t(M_{\underline{u}})_\q \subseteq H^i_{W^{-1}I}(W^{-1}R)_{\underline{u}}$ is a finitely generated $A_\q$-module and $E_A(A/\q)_\q \cong E_{A_\q}(k(\q))$ can not occur as its direct summand.
\end{proof}

\begin{remark}
The above result is interesting only when $A$ has infinitely many height one prime ideals.
\end{remark}

\section{Understanding the torsion-free part of the components}

The following statement is an immediate consequence of some well-known
results.

\begin{theorem}\label{terai}
Fix $\underline{u} \in \Z^n$. Let $K$ be the fraction field of $A$. Then $M_{\underline{u}} \otimes_A K \cong K^{\alpha(\underline{u})}$ for some finite $\alpha(\underline{u}) \geq 0$. Moreover, $\{\alpha(\underline{u}) \mid \underline{u} \in \Z^n\}$ is a finite set.
\end{theorem}

\begin{proof}
 Since $R=A[X_1, \ldots, X_n] \hookrightarrow V:=K[X_1, \ldots, X_n]$ is a flat extension, by the \emph{flat base change theorem}, $H^i_I(R) \otimes V \cong H^i_{IV}(V)$. Observe that $IV$ is a usual monomial ideal in $V$ and
 \[H^i_{IV}(V)_{\underline{u}}=H^i_I(R)_{\underline{u}} \otimes_A K=M_{\underline{u}} \otimes K.\]
 By \cite[Corollary 3.3]{KY2}, graded components of $H^{i}_{J}(V)$
are finite dimensional $K$-vector spaces for any square-free monomial ideal $J$ in $V$ and $i\geq 0$. In fact, this statement is independent of the characteristic of the field $K$. Since, $H^i_{J}(M)=H^i_{\sqrt{J}}(M)$ for any $V$-module $M$ and ideal $J \subseteq V$, the result follows.

From N. Terai’s formula in \cite{Ter} one has
\[H^i_I(V)_{\underline{u}} \cong H^i_I(V)_{\underline{u}+e_i}\]
for all $\underline{u} \in \Z^r$ such that $u_i \neq -1$, see also \cite[Theorem 2.1 b)]{MM}, \cite[2.3]{Mon_LC}.
It follows that for any $\underline{a}\in \{0,-1\}^n$,
\begin{equation}\label{block_rel}
H^i_I(V)_{\underline{u}} \cong H^i_I(V)_{\underline{a}}
\end{equation}
for each $\underline{u} \in \Z^n$ with $u_i \geq 0$ if $a_i=0$ and $u_i \leq -1$ if $a_i=-1$ for $i=1, \ldots, n$. Hence
\[\left\{\alpha(\underline{u}) \mid \underline{u} \in \Z^n\right\}=\left\{\alpha(\underline{a}) \mid \underline{a} \in \{0,-1\}^n\right\},\]
which is indeed a finite set.
\end{proof}

\section{Local cohomology modules with support
at usual monomial ideals}

All the results in this section can be obtained from some known results when $A=k$ is a field. We are stating them to convey an overview and for the ease of understanding of the readers.

As $A$ has characteristic zero so $\chars K=0$. We further have a natural inclusion
$R:=A[X_1, \ldots, X_n] \hookrightarrow V:=K[X_1, \ldots, X_n]$.

The following result is proved in \cite[Theorem 3.1.4 and Theorem 3.3.3]{MoRoSa}. Here, we provide a self-contained proof without introducing any new notations.

\begin{theorem}\label{priDec}
 Let $I \subseteq R=A[X_1, \ldots, X_n]$ be a usual monomial ideal. Then $I=\cap_{i=1}^m Q_i$, where $Q_i=(X_{i_1}^{a_1}, \ldots, X_{i_k}^{a_k})$ for some $a_j \geq 1$.
\end{theorem}

We first prove a lemma. Following the notations in \cite{HerHib}, we set
\[\mathrm{Mon}(V)=\{X_1^{\xi_1}\cdots X_n^{\xi_n}\mid \xi_i \geq 0 \mbox{ for } i=1, \ldots, n\}.\]
For any polynomial $f$ in $V$, we write $f=\sum_{U \in \mathrm{Mon}(V)} \alpha_U U$ where $\alpha_U \in K$. Then the \emph{ support of} $f$, denoted by $\mathrm{supp}(f)$, is defined as
\[\mathrm{supp}(f):=\{U \in \mathrm{Mon}(V) \mid \alpha_U \neq 0\}.\]

\begin{lemma}\label{main_alt}
 Let $I$ be a usual monomial ideal in $R$. Then $IV \cap R=I$.
\end{lemma}

\begin{proof}
 We know that $I \subseteq IV \cap R$. Conversely, let $f \in IV \cap R$. Suppose that $I=(U_1, \ldots, U_c)R$ for some monomials $U_1, \ldots, U_c$ in $R$. Then $IV=(U_1, \ldots, U_c)V$ is also a usual monomial ideal. Thus by \cite[Corollary 1.1.3]{HerHib}, $\supp(f) \subseteq IV$. So for any monomial $U \in \supp(f)$, by \cite[Proposition 1.1.5]{HerHib}, $U=wU_i$ for some monomial $w$ in $V$. Notice that $w:=X_{i_1}^{a_1}\cdots X_{i_m}^{a_m}$ is also a usual monomial in $R$. Hence $U \in I$, i.e., $\supp(f) \subseteq I$.
 Thus $f \in I$. The result follows.
\end{proof}

\begin{proof}[Proof of Theorem \ref{priDec}]
 Suppose that $\{U_1, \ldots, U_c\}$ is a usual monomial generating set of $I$. Then $IV=(U_1, \ldots, U_c)V$ is a usual monomial ideal in $V$. Hence by \cite[Theorem 1.3.1]{HerHib},
 \[IV=\bigcap_{i=1}^r Q_i,\]
 where $Q_i=(X_{i_1}^{a_1}, \ldots, X_{i_k}^{a_k})V$ for $1 \leq i_1<i_2<\cdots<i_k \leq n$. Notice that $Q_i'=(X_{i_1}^{a_1}, \ldots, X_{i_k}^{a_k})R$ is a usual monomial ideal in $R$ and $Q_i'V=Q_i$. Therefore, $Q_i \cap R=Q'_i$ by Lemma \ref{main_alt}. Thus we get that
 \[I=IV \cap R= \left(\bigcap_{i=1}^r Q_i \right) \cap R =\bigcap_{i=1}^r \left(Q_i \cap R\right)=\bigcap_{i=1}^r Q_i'.\]
 The result follows.
\end{proof}

\begin{note*}
Recall that if $I=\bigcap_{i=1}^m Q_i$ is an irredundant primary decomposition, then
\[\Ass R/I=\{\sqrt{Q_i} \mid i =1, \ldots, r\}.\]
Hence the associated primes of $I$ are monomial prime ideals.
\end{note*}

Let $C$ be a commutative Noetherian ring and $W=C[Y_1, \ldots, Y_m]$ be a polynomial ring over $C$. We set $\m=(Y_1, \ldots, Y_m)W$.
In the forthcoming theorem, we are going to use the following well-known result:
\begin{equation}\label{LC-inv}
H^i_\m(W)=
\begin{cases}
 C[Y_1^{-1}, \ldots, Y_m^{-1}](-1,-1,\cdots,-1)  & \mbox{ if } i = m \\
 0 & \mbox{ otherwise }.\\
\end{cases}
\end{equation}

\begin{note}\label{primeLC}
 Let $P$ be a monomial prime ideal in $R$, that is, $P=(X_{i_1}, \ldots, X_{i_t})$ for some $1 \leq i_1< \cdots < i_t \leq n$. We relabel the variables and assume that $P=(X_1, \ldots, X_t)$. Now we put $W=A[X_1, \ldots, X_t]$. Observe that $W \hookrightarrow W[X_{t+1}, \ldots, X_n] = R$ is a flat extension. Therefore,
\begin{align*}
 H^i_{P}(R)&\cong H^i_P(W) \otimes_W R\\
 &\cong H^i_P(W) \otimes_W \left(W \otimes_A A[X_{t+1}, \ldots, X_n]\right)\\
 &\cong H^i_P(W) \otimes_A A[X_{t+1}, \ldots, X_n]\\
&= \begin{cases}
 A[X_1^{-1}, \ldots, X_t^{-1}, X_{t+1}, \ldots, X_n](-1,-1,\cdots,-1,0,\cdots,0) & \mbox{ if } i = t \\
 0 & \mbox{ otherwise }.
\end{cases}
\end{align*}
(In the above formula we have shift of $-1$ in the first $t$ co-ordinates and no shift in the rest).
due to \eqref{LC-inv}.
\end{note}

\begin{proposition}\label{fin-gen}
 Suppose that $I$ is a usual monomial ideal. Then
$M_{\underline{u}}$ are finitely generated $A$-modules for all $\underline{u} \in \Z^n$.
\end{proposition}

\begin{proof}
Let $\Ass R/I=\{P_1, \ldots, P_s\}$. By Theorem \ref{struThm}, $P_i=(X_{i_1}, \ldots, X_{i_r})$ for $1 \leq i_1< \cdots <i_r \leq n$. Recall that $H^i_I(R)=H^i_{\sqrt{I}}(R)$ for each $i \geq 0$. After re-arranging, we write
\[\sqrt{I}=P_1 \cap P_2 \cap \cdots \cap P_s,\]
where $\hgt P_1 \leq \hgt P_2 \leq \cdots \leq \hgt P_s$. We use descending induction on height $h$ of the ideal $I$. Suppose that $h=n$. Then $s=1$ and $P_1=(X_1, \ldots, X_n)$. So $M_{\underline{u}}$ is finitely generated due to \eqref{LC-inv}.

It is now sufficient to show that the result is true when $h=r<n$ assuming that it is true for ideals with height greater than or equal to $r+1$.
To prove this statement, we use ascending induction on $s$, the number of associated primes of $I$. Let $s=1$. In light of Note \ref{primeLC}, graded components $H^i_{P_1}(R)$ are finitely generated. Now suppose that the result is known for $s-1$. We set $P':=P_2 \cap \cdots \cap P_s$. Then $\sqrt{I}=P_1 \cap P'$.
Consider the long exact sequence
\[ \cdots \to H^i_{P_1+P'}(R) \to H^i_{P_1}(R) \oplus H^i_{P'}(R) \to H^i_{\sqrt{I}}(R) \to H^{i+1}_{P_1+P'}(R) \to \cdots.\]
Since
$\hgt(P_1+P') \geq \hgt P_1+1=\hgt I+1=r+1$,
by our assumption graded components of
$H^i_{P_1+P'}(R)$ are finitely generated. Again by induction hypothesis on $s$, graded components of
$H^i_{P'}(R)$ are finitely generated. It follows that
graded components of $H^i_I(R)$ are finitely generated.
\end{proof}

\section{Flatness of the torsion-free part}\label{flat_free} Consider the exact sequence
 \[0 \to t\left(M_{\underline{u}}\right) \to M_{\underline{u}} \xrightarrow{f} M_{\underline{u}} \otimes K.\]
 By $t(-)$ we denote the torsion part of some module. We set $f(M_{\underline{u}})=\overline{M_{\underline{u}}}$. Then we get the short exact sequence
 \[0 \to t\left(M_{\underline{u}}\right) \to M_{\underline{u}} \to \overline{M_{\underline{u}}} \to 0,\]
which induces the short exact sequence
 \[0 \to t\left(M_{\underline{u}}\right)_\p \to \left(M_{\underline{u}}\right)_\p \to \left(\overline{M_{\underline{u}}}\right)_\p \to 0\]
 for each prime ideal $\p$ in $A$. Take $\p \in \mathcal{Z}$, defined in \eqref{assPrime} and fix it. We put $B:=A_\p$ and $\n:=\p A_\p$. Since $(B, \n)$ is a DVR, any nonzero element is of the form $u \pi^r$, where $u$ is a unit in $B$ and $\n=(\pi)$. From Proposition \ref{tor_sub_fun} and Theorem \ref{struThm}, we have the following:
 \[t\left(M_{\underline{u}}\right)
=\bigoplus_{\p \in \mathcal{Z}} t\left(M_{\underline{u}}\right)_{\p} \quad
\mbox{ and } \quad t\left(M_{\underline{u}}\right)_\p= E(A/\p)^{\ell(\underline{u})} \oplus \left(\bigoplus_{j=1}^{t(\underline{u})} B/\n^{\beta_j(\underline{u})}\right)\]
for each $\p \in \mathcal{Z}$. Together we get that
\begin{equation}\label{torsion}
t\left(M_{\underline{u}}\right)=\left(\bigoplus_{i=1}^r E(A/\p_i)^{\ell_i(\underline{u})} \right) \oplus \mbox{ a finite length module}.
\end{equation}

\begin{proposition}\label{flat}
For each $\underline{u} \in \Z^r$,
$\overline{M_{\underline{u}}}$ is a flat $A$-module.
\end{proposition}

\begin{proof}
 We need to show that
\[\Tor^A_1(\overline{M_{\underline{u}}}, L) =0\]
 for all finitely generated $A$-module $L$. Now   for all
prime ideals $\q$ in $A$, we have
 \[\Tor^A_1(\overline{M_{\underline{u}}}, L) \otimes_A A_\q \cong \Tor^{A_\q}_1(\overline{\left(M_{\underline{u}}\right)}_{\q}, L_\q).\]
 If $\q=0$, then $A_\q$ is the fraction field of $A$ and hence $\Tor^{A_\q}_1(\overline{\left(M_{\underline{u}}\right)}_{\q}, L_\q)=0$. So it is enough to prove the statement for $V:=M_{\q}$ with graded components $V_{\underline{u}}=M_{\underline{u}}\otimes_A B$ over $B:=A_\q$ for any nonzero prime ideal $\q$ in $A$. Since the completion $T:=\widehat{B}$ of $B$ with respect to $\q B$ is a faithfully flat $B$-module. In light of \cite[Theorem 7.2]{Mat}, it is enough to show that $\Tor^{B}_1(\overline{V_{\underline{u}}},L) \otimes_B \widehat{B}\cong \Tor^{\widehat{B}}_1(\widehat{\overline{V_{\underline{u}}}},\widehat{L})=0$ for all finitely generated $B$-module $L$. From Theorem \ref{struThm} we have $\overline{\widehat{V_{\underline{u}}}} \cong T^{a(\underline{u})} \oplus K_\q^{b(\underline{u})}$ for finite $a(\underline{u}),b(\underline{u}) \geq 0$. Since $K_\q =T_{(0)}$ is a flat $T$-module, the result follows.
\end{proof}

\section{Splitting}

Fix $\underline{u}\in \Z^n$. We first recall the short exact sequence
 \begin{equation}\label{eq_ses}
 0 \to t\left(M_{\underline{u}}\right) \to M_{\underline{u}} \to \overline{M_{\underline{u}}} \to 0,
 \end{equation}
 which we obtained in
Section \ref{flat_free}. In consequence of Proposition \ref{fin-gen}, \ref{flat}  we get the following.

\begin{corollary}\label{mono_split}
If $I$ is a usual monomial ideal, then the short exact sequence \eqref{eq_ses} splits.
\end{corollary}

\begin{proof}
Recall that over a Dedekind domain any finitely generated, torsion-free module is projective. From Proposition \ref{fin-gen}, \ref{flat}  it thus follows that $\overline{M_{\underline{u}}}$ is a projective module. Hence
\[\Ext^1_{A} \left(\overline{M_{\underline{u}}}, t\left(M_{\underline{u}}\right)\right)=0.\]
So \eqref{eq_ses} splits.
\end{proof}

We don't know whether \eqref{eq_ses} splits when $I$ is a $\mathfrak{C}$-monomial ideal. However, we now show that \eqref{eq_ses} splits when $A$ is a PID. In particular, \eqref{eq_ses} splits locally. This is the main result in this section.

\begin{proposition}\label{split}
 If we further assume that $A$ is a PID, then the short exact sequence \eqref{eq_ses} splits.
\end{proposition}


To prove Proposition \ref{split}, we use a property of \emph{Pontrjagin dual}, which is defined and discussed for
the ring of integers in \cite[Definition 3.2.3, Proposition 3.2.4]{Wei}. Although the same
treatment works for a PID, we repeat the proof here for the reader's convenience.

Suppose that $A$ is a PID. Then $A$ has an injective resolution
\[0 \to K \to K/A \to 0,\]
where $K$ denotes the quotient field of $A$. By \cite[Corollary 3.1.5]{BH}, both $K$ and $K/A$ are injective modules. Therefore,
\[K/A \cong \bigoplus_{\substack{\p \in \Spec(R)\\ \p \neq 0}} E_A(A/\p).\]
The Pontrjagin dual of an $A$-module $N$, denoted by $N^*$, is defined as
\[N^*:=\Hom_A(N, K/A).\]

\begin{lemma}\label{PDualInj}
 If $N$ is a flat $A$-module, then $N^*$ is an injective $A$-module.
\end{lemma}

\begin{proof}
 Let $M$ be a $A$-module. By the \emph{tensor-hom adjunction},
 \[\Hom_A(M, N^*)=\Hom_A(M, \Hom_A(N, K/A)) \cong \Hom_A(M \otimes_A N, K/A)=(M \otimes_A N)^*.\]
By \emph{Baer's criterion} (see \cite[Proposition 3.1.2 d)]{BH}), $N^*$ is an injective module if and only if for any ideal $I$ in $A$ the map
\[\Hom_A(A, N^*) \to \Hom_A(I, N^*),\]
i.e., $(A \otimes_A N)^*\to (I \otimes_A N)^*$ is a surjective map. Since $K/A$ is an injective $A$-module, in view of \cite[Proposition 3.1.2]{BH}, it is enough to show that $I \otimes_A N \to A \otimes_A N$ is an injective map. The latter condition holds if and only if $N$ is a flat $A$-module, see \cite[Theorem 7.7]{Mat}.
\end{proof}

\begin{proof}[Proof of Proposition \ref{split}]
Fix $\underline{u} \in \Z^n$. Using Theorem \ref{struThm}, we get the short exact sequence
\[0 \to \bigoplus_{i=1}^r \left(E_A(A/\p_i)^{\ell_i} \oplus \bigg(\bigoplus_{j=0}^{t_i} \frac{A}{\p_i^{\beta_{ij}}}\bigg)\right) \to M_{\underline{u}} \to \overline{M_{\underline{u}}} \to 0\]
for some finite $\ell_i$ and $\beta_{ij}$. Since $\Ext^1_A\left(\overline{M_{\underline{u}}}, E_A(A/\p_i)^{\ell_i}\right)=0$, we now show that
\[\Ext^1_A(\overline{M_{\underline{u}}}, A/\p_i^{\beta_{ij}})=0.\]
Fix $\p \in \mathcal{Z}$. Since $A$ is a PID, $\p=(\pi)$ for some nonzero element $\pi$ in $A$. Consider the minimal free resolution
\begin{equation}\label{minReso}
 0 \to A \xrightarrow{\pi^{\beta}} A \to A/\p^{\beta} \to 0
\end{equation}
of $A/\p^{\beta}$. Let $\widehat{(-)}$ denote the completion of any $A$-module with respect to the maximal ideal $\p$ of $A$. As $A/\p^{\beta}$ is an Artinian $A$-module  with $\Ass_A A/\p^{\beta}=\{\p\}$ so we have $\widehat{A/\p^{\beta}} \cong A/\p^{\beta}$. Observe that
\[\Hom_A(A/\p^{\beta}, E_A(A/\p)) \cong E_{A/\p^{\beta}}(A/\p) \cong A/\p^{\beta},\]
since $A/\p^{\beta} \cong A/(\pi^{\beta})$ is Gorenstein and Artinian. Applying $\Hom_A(-, E_A(A/\p))$ to \eqref{minReso}, we get the short exact sequence
\begin{equation}\label{injReso}
0 \to A/\p^{\beta} \to E_A(A/\p) \xrightarrow{\cdot \pi^{\beta}} E_A(A/\p) \to 0.
\end{equation}
As $E_A(A/\p)$ is an indecomposable injective $A$-modules, it follows that
\begin{equation}\label{inj}
 0 \to E_A(A/\p) \xrightarrow{\pi^{\beta}} E_A(A/\p) \to 0
\end{equation}
is an injective resolution of $A/\p^{\beta}$ as an $A$-module. Applying $\Hom_A(\overline{M_{\underline{u}}}, -)$ to \eqref{inj}, we get the complex
\begin{equation}\label{hom-inj}
0\to \Hom_A(\overline{M_{\underline{u}}}, E_A(A/\p)) \xrightarrow{\pi^{\beta}}\Hom_A(\overline{M_{\underline{u}}}, E_A(A/\p)) \to 0.
\end{equation}
From Proposition \ref{flat} we have $\overline{M_{\underline{u}}}$ is a flat $A$-module. So the \emph{Pontrjagin dual}
\[\overline{M_{\underline{u}}}^*:=\Hom_A(\overline{M_{\underline{u}}}, K/A)=\Hom_A\bigg(\overline{M_{\underline{u}}}, \bigoplus_{\substack{\q \in \Spec(R)\\ \q \neq 0}} E_A(A/\q)\bigg) \] is an injective $A$-module by Lemma \ref{PDualInj}.
Since $E_A(A/\p)$ is a direct summand of $K/A$ so $\Hom_A(\overline{M_{\underline{u}}}, E_A(A/\p))$ is a direct summand of $\overline{M_{\underline{u}}}^*$. Thus  $\Hom_A(\overline{M_{\underline{u}}}, E_A(A/\p))$ is an injective $A$-module. Hence by Matlis theory, \[\Hom_A(\overline{N_{\underline{u}}}, E_A(A/\p)) \cong E_A(A/\p)^w \oplus \left( \bigoplus_{\substack{\q \in \Spec(R)\\ \q \neq \p}} E_A(A/\q)^{v(\q)} \right)\]
for some $v(\q), w$, see \cite[Proposition 3.2.4]{Wei}. Observe that the multiplication map by $\pi^{\beta}$ in \eqref{hom-inj} is an isomorphism on $E_A(A/\q)$ for every prime ideal $\q \neq \p$ and is an onto map on $E_A(A/\p)$ by \eqref{injReso}. It therefore follows that $\Ext^1_A(\overline{N_{\underline{u}}},A/\p^{\beta})=0$.
\end{proof}

As a sequel to Proposition \ref{split} we obtain the following.
\begin{corollary}\label{split_local}
 The induced short exact sequence
 \[0 \to t\left(M_{\underline{u}}\right)_{\p} \to \left(M_{\underline{u}}\right)_{\p} \to \left(\overline{M_{\underline{u}}}\right)_\p \to 0\]
 splits for all prime ideal $\p$ in $A$.
\end{corollary}

\begin{proof}
The statement is trivial at $\p=0$.
Therefore, we take a nonzero prime ideal $\p$ of $A$ and fix it. We set $W=A \backslash \p$. Clearly, $W^{-1}R \cong A_\p[X_1, \ldots, X_n]$. Since $A_\p$ is a DVR, it is a PID. We set $N:=M \otimes_A A_\p \cong H^i_{W^{-1}I}(W^{-1}R)$. Observe that $N_{\underline{u}}=M_{\underline{u}} \otimes_A A_\p \cong (M_{\underline{u}})_\p$ for all $\underline{u} \in \Z^n$. Hence the result follows from Proposition \ref{split}.
\end{proof}

\begin{remark}
 Note that $\overline{M_{\underline{u}}}$ is not finitely generated over $A_\p$. Otherwise, if $\overline{M_{\underline{u}}}$ is finitely generated, then $\overline{N_{\underline{u}}}$ in Theorem \ref{struThm} is finitely generated. But $K_\p$ can occur as a direct summand of $\overline{N_{\underline{u}}}$, which is not finitely generated as a $T$-module. Therefore, $\Ext^1(\overline{M_{\underline{u}}}, t\left(M_{\underline{u}}\right))$ may not commute with localization. So \eqref{eq_ses} not necessarily splits for arbitrary $\mathfrak{C}$-monomial ideal $I$ despite Proposition \ref{split}.
\end{remark}

As a consequence of the Theorem \ref{struThm}, we
get finiteness of the Bass numbers of the graded components. In fact, we can explicitly compute the Bass numbers.

\begin{corollary}[Bass numbers]\label{bass}
Fix $\underline{u} \in \Z^n$. The Bass numbers $\mu_j(\q, M_{\underline{u}})$ of $M_{\underline{u}}$ are finite for each prime ideal $\q$ in $A$ and every $j \geq 0$.
\end{corollary}

\begin{proof}
Since $A_\p$ is regular, for any $A$-module $L$, prime ideal $\p$ in $A$ and $j \geq 0$,
\begin{equation}\label{bassCompletion}
\mu_j(\p, L)
=\mu_j(\widehat{\p A_\p}, \widehat{L_\p})
\end{equation}
by \cite[2.1.1]{Lyu2000}.
Fix a prime ideal $\p$. We set
$T:=\widehat{A_\p}$ and put $N:=M \otimes_A T$. Clearly, $N_{\underline{u}}=M_{\underline{u}}\otimes_A T=\widehat{\left(M_{\underline{u}}\right)_\p}$ for all $\underline{u} \in \Z^n$. Denote $\n:= \p T=\widehat{\p A_\p}=(\pi)$. Then
\[\mu_j(\p, M_{\underline{u}})=\mu_j(\n, N_{\underline{u}})=\dim_{k(\p)} \Ext^j_{T}\left(k(\p), N_{\underline{u}}\right),\]
where $k(\p):=T/\n= \widehat{\frac{A_\p}{\p A_\p}} \cong A/\p$. Further, observe that the Koszul complex
\[0 \to T \xrightarrow{\cdot \pi} T \to T/(\pi):=k(\p) \to 0\]
is a minimal free resolution of $k(\p)$. So $\mu_j(\n, N_{\underline{u}})=0$ for $j \geq 2$. By Corollary \ref{split_local}, $N_{\underline{u}} \cong \Gamma_\p\left(N_{\underline{u}}\right) \oplus \overline{N_{\underline{u}}}$ and hence
\begin{equation}\label{ext_rel}
 \Ext^j_{T}\left(k(\p), N_{\underline{u}}\right)= \Ext^j_{T}\left(k(\p),\Gamma_\p\left(N_{\underline{u}}\right) \oplus \overline{N_{\underline{u}}}\right).
\end{equation}
for every $j \geq 0$.
Moreover, from Theorem \ref{struThm} we have
\begin{align*}
 &\Gamma_\p\left(N_{\underline{u}}\right)=E_T(k(\p))^{\ell(\underline{u})} \oplus \left( \bigoplus_{j=1}^{t(\underline{u})} T/\p^{\beta_j(\underline{u})}T \right) \quad \mbox{ for } \p \in \mathcal{Z},
\mbox{ defined in \eqref{assPrime}}\\
&\overline{N_{\underline{u}}}=T^{a(\underline{u})} \oplus K_\p^{b(\underline{u})} \hspace{4.5cm} \mbox{ for } \p \in \Spec(A) \backslash \{0\},
\end{align*}
where $\ell(\underline{u}), t(\underline{u}), a(\underline{u}), b(\underline{u}) \geq 0$ are finite. The result follows.
 \end{proof}


\s By a \emph{block} $\mathcal{B}(n)$, we mean that there is a subset $U$ of $\{1, \ldots, n\}$ such that
\[\mathcal{B}(n) = \{\underline{u} \in \Z^n \mid u_i \geq 0 \mbox{ if } i \in U \mbox{ and } u_i \leq -1 \mbox{ otherwise}\}.\]

\begin{corollary}[Tameness]
Fix $\underline{u} \in \Z^n$. Suppose that $M_{\underline{u}} \neq 0$ and it has a torsion-free element. Then there is a block $\mathcal{B}(n)$ such that $M_{\underline{v}} \neq 0$ and it has a torsion-free element for all $\underline{v} \in \mathcal{B}(n)$.
\end{corollary}

\begin{proof}
In view of \cite[Lemma 2.1]{BBLSZ}, $N:=M \otimes_A K$ is $\mathcal{D}(R \otimes_A K, K)$ module. We set $V:=R \otimes_A K \cong K[X_1, \ldots, X_n]$. Then $N \cong H^i_{IV}(V)$. If $M_{\underline{u}}$ has a torsion-free element, then $N_{\underline{u}} \neq 0$. Since $N$ is an $\Z^n$-graded generalized Eulerian $\D(V, K)$-module, by Terai's result \cite{Ter}, there is a block $\mathcal{B}(n)$ by \eqref{block_rel} such that $N_{\underline{v}} \neq 0$ and hence $M_{\underline{v}} \neq 0$. Moreover, $M_{\underline{v}}$ has a torsion-free element. The result follows.
\end{proof}

\begin{note*}
 From the above result, we get that if $0 \in \Ass_A M_{\underline{u}}$ for some $\underline{u} \in \Z^n$, then there is a block $\mathcal{B}(n)$ such that $0\in \Ass_A M_{\underline{v}}$.
\end{note*}

\section{Examples}

In this section, we consider $A=\Z$. We first discuss an example where $\beta_j(\underline{u}) \neq 0$ in the expression of $t(M_{\underline{u}})$ as in Theorem \ref{struThm} for some $\underline{u} \in \Z^n$.

\begin{example}\label{eg1}
 Suppose that $R=\Z[X]$ and $I=(2X)$. Then it is easily seen that $(2X)=(2) \cap (X)$. For instance, $(2X) \subseteq (2) \cap (X)$. Now take $f \in (2) \cap (X)$. Then $f=g(X) X$ and $2 \mid g(X) X$. Since $R$ is a UFD and $\gcd(2, X)=1$, we get $2 \mid g(X)$. Hence $f \in (2X)$. Now consider the Mayer-Vietoris sequence
 \[\cdots \to H^1_{(2, X)}(R) \to H^1_{(2)}(R) \oplus H^1_{(X)}(R) \to H^1_{(2X)}(R) \to H^2_{(2, X)}(R) \to 0.\]
 As $\hgt (2, X)=2$ so we have $H^1_{(2, X)}(R)=0$. Further notice that
 \[H^1_{(2)}(R) \cong H^1_{(2)}(\Z)[X]\cong E_\Z(\Z/2 \Z)[X] \quad \mbox{ and } \quad H^1_{(X)}(R) \cong \Z[X^{-1}](-1).\]
 Hence
 \[H^1_{(2X)}(R)_u \supseteq
 \begin{cases}
  E_\Z(\Z/2 \Z) \quad \mbox{ if } u \geq 0\\
  \Z  \quad \mbox{ otherwise }.
 \end{cases}
\]
\end{example}

We now give an example where the torsion-free part $\overline{M_{\underline{u}}}$ is not finitely generated for some $\underline{u} \in \Z^n$.

\begin{example}
 Suppose that $R=\Z[X, Y]$ and $I=(2X, Y)$. We set $\overline{R}=\frac{\Z}{2\Z}[X, Y]$. Consider the short exact sequence
 \[0 \to R \xrightarrow{\cdot 2} R \to\overline{R} \to 0,\]
 which induces a long exact sequence
 \[\cdots \to H^1_{(Y)}(\overline{R}) \to H^2_{(2X, Y)}(R) \xrightarrow{\cdot 2} H^2_{(2X, Y)}(R) \to H^2_{(Y)}(\overline{R}) \to \cdots.\]
 Observe that $H^2_{(Y)}(\overline{R})=0$ as $\hgt (Y)=1$ and $H^1_{(Y)}(\overline{R})_{(u_1, u_2)}$ is a finite dimensional vector space over $\Z/2 \Z$ for each pair ${(u_1, u_2)} \in \Z^2$. We put $M:= H^2_{(2X, Y)}(R)$. Then we have $M_{(u_1, u_2)} \xrightarrow{\cdot 2} M_{(u_1, u_2)} \to 0$ for all ${(u_1, u_2)} \in \Z^2$. Moreover, from the commutative diagram
 \[
\xymatrix@C=0.75em@R=0.75em{
	M_{(u_1, u_2)} \ar[r] \ar[d]^2 & \overline{M_{(u_1, u_2)}} \ar[d]^2 \ar[r] & 0\\
	M_{(u_1, u_2)} \ar[r] \ar[d]& \overline{M_{(u_1, u_2)}} \ar[r] & 0\\
	0 & &
	}\]
and using the snake lemma, we get that $\overline{M_{(u_1, u_2)}} \xrightarrow{\cdot 2} \overline{M_{(u_1, u_2)}} \to 0$ for all ${(u_1, u_2)} \in \Z^2$. Recall that $\overline{M_{(u_1, u_2)}}$ is a torsion-free $\Z$-module. If it is finitely generated, then it is a free $\Z$-module. But in that case the surjective map $\overline{M_{(u_1, u_2)}} \xrightarrow{\cdot 2} \overline{M_{(u_1, u_2)}}$ is an isomorphism, which leads to a contradiction if $\overline{M_{(u_1, u_2)}}  \neq 0$. We note that
$$\overline{M}\otimes_{\Z} \mathbb{Q}  = M \otimes_{\Z} \mathbb{Q} = H^2_{(X, Y)}(\mathbb{Q}[X,Y]) = \mathbb{Q}[X^{-1}, Y^{-1}](-1, -1).$$
\end{example}

Right now, we do not have any answer to the following problems.
\begin{question}
Does there exist some $R$ and $\mathfrak{C}$- monomial $I \subseteq R$ such that both $a(\underline{u})$ and $b(\underline{u})$ in Theorem \ref{struThm} are nonzero?
\end{question}

\begin{question}
Does there exist any Dedekind domain $A$ and a $\mathfrak{C}$-monomial ideal $I \subseteq R$ such that the short exact sequence \eqref{eq_ses} does not split for some $\underline{u}$?
\end{question}

\end{document}